\documentclass[a4paper]{article}
\usepackage{amsmath,amsfonts,amsthm, amssymb}
\usepackage[all]{xy}
\newtheorem{theorem}{Theorem}[section] 
\newtheorem{lemma}[theorem]{Lemma} 
\newtheorem{proposition}[theorem]{Proposition}

\newtheorem{definition}[theorem]{Definition}
\newenvironment{proofof}[1]{\normalsize {\it Proof of #1}:}{{\hfill $\Box$}}

\newcommand{\Z}{{\mathbb Z}}
\newcommand{\N}{{\mathbb N}}

\newcommand{\Sym}{{\mathcal S}}

\newenvironment{mylist}{\begin{list}{}{
\setlength{\parskip}{0mm}
\setlength{\topsep}{2mm}
\setlength{\parsep}{0mm}
\setlength{\itemsep}{0.5mm}
\setlength{\labelwidth}{7mm}
\setlength{\labelsep}{3mm}
\setlength{\itemindent}{0mm}
\setlength{\leftmargin}{12mm}
\setlength{\listparindent}{6mm}
}}{\end{list}}
\title{Sofic groups: graph products and graphs of groups}
\author{Laura Ciobanu, Derek F. Holt and Sarah Rees}
\date{Warwick, 21st June 2013}
\begin{document}
\maketitle
\begin{abstract}
We prove that graph products of sofic groups are sofic,
as are graphs of groups for which vertex groups are sofic and
edge groups are amenable.
\end{abstract}

\noindent 2010 Mathematics Subject Classification: 20F65, 37B05.

\noindent Key words: sofic, graph products, free and direct
products, groups of graphs.

\bigskip

\section{Introduction}
We prove the following results.
\begin{theorem}
\label{thm:graphprod}
A graph product of sofic groups is sofic.
\end{theorem}

\begin{theorem}
\label{thm:graphgroups}
The fundamental group of a graph of groups is sofic if each vertex group is sofic and each edge group is amenable.
\end{theorem}

Theorem~\ref{thm:graphprod} generalises Theorem 1 of \cite{ElekSzabo}, and our
proof is based on ideas used in the proof of that theorem. 
Theorem~\ref{thm:graphgroups} is an  extension of the result
that free products of sofic groups amalgamated over amenable subgroups
are sofic, proved independently in \cite[Theorem 1]{ElekSzabo2} and
\cite[Corollary 2.3]{Paunescu}; most of the argument needed to extend
the result is already found in \cite[Corollary 3.6]{CollinsDykema}.

The term sofic groups is attributed to Weiss \cite{Weiss}, and applied
to a definition due to Gromov \cite{Gromov};
this is a class of groups which, together with the related class of hyperlinear
groups, has inspired much recent study, through its connections to a variety
of different mathematical areas.
A very useful introduction to sofic groups is provided by \cite{Pestov}.
There are many open questions, including the question of whether all 
groups are sofic.

A number of quite distinct, but equivalent, definitions exist for sofic
groups, and are proved equivalent in \cite{Pestov}. The definition in \cite{Weiss} for finitely generated groups involves
finite subsets of the Cayley graph of the group, and is essentially the same
as the definition in~\cite{Gromov} of the Cayley graph being
{\em initially subamenable}.
An alternative and  equivalent definition of \cite{Pestov} defines
a group  to be sofic if it embeds as a subgroup in an ultraproduct
of symmetric groups. 
Another (equivalent) definition, found in \cite{ElekSzabo}, is phrased in
terms of quasi-actions. We shall work with a variation of that
definition, given below as Definition~\ref{defn:sofic}; we phrase it in terms
of (what we call) {\em special} quasi-actions.
That this is equivalent to the definition of \cite{ElekSzabo} (and hence to the others) follows from \cite[Lemma 2.1]{ElekSzabo}.

For a finite set $A$, let $\Sym(A)$ be the group of all permutations of $A$.
For $\epsilon >0$, we say that two elements $f_1,f_2$ of $\Sym(A)$ are
$\epsilon$-similar  if the number
of elements $a\in A$ for which $f_1(a) \neq f_2(a)$ is at most $\epsilon|A|$.
Note that for $\epsilon\geq 1$ this condition is always satisfied.

\begin{definition}
\label{defn:qaction}
Suppose that $G$ is a group, $\epsilon > 0$ a real number and
$F \subseteq G$ a finite subset of $G$.
A special $(F,\epsilon)$-quasi-action of $G$ on a finite set $A$ is a function
$\phi:G \rightarrow \Sym(A)$ with the following properties:
\begin{mylist}
\item[(a)] $\phi(1)=1$;
\item[(b)] $\forall g \in G$,$\quad\phi(g)^{-1}=\phi(g^{-1})$;
\item[(c)] for $g \in F\setminus\{1\}$, $\phi(g)$ has no fixed points;
\item[(d)] for $g_1,g_2 \in F$ the map $\phi(g_1g_2)$ is 
$\epsilon$-similar to $\phi(g_1)\phi(g_2)$.
\end{mylist}
\end{definition}
For $a \in A, g \in G$, we 
write $a^{\phi(g)}$ for the image of $a$ under $\phi(g)$.

\begin{definition}
\label{defn:sofic}
A group $G$ is {\em sofic} if, for each number $\epsilon \in (0,1)$
and any finite subset $F \subseteq G$, $G$ admits a special
$(F,\epsilon)$-quasi-action.
\end{definition}

It is immediate from the definition that a group is sofic precisely if every one
of its finitely generated subgroups is sofic.
We note at this stage also the following elementary result, which will be
useful to us later.
\begin{lemma}\label{lem:product}
Let $\phi_i$ be special $(F,\epsilon)$-quasi-actions of $G$ on $A_i$ for
$1 \le i \le n$, let $A = A_1 \times \cdots \times A_n$, and define
$\phi: G \to \Sym(A)$ by 
$(a_1,\ldots,a_n)^{\phi(g)} = (a_1^{\phi_1(g)},\ldots,a_n^{\phi_n(g)})$.
Then $\phi$ is a special $(F,n\epsilon)$-quasi-action.
\end{lemma}
\begin{proof} The
conditions (a), (b) and (c) of the definition are straightforward to check for $\phi$. The
equality
$(a_1,\ldots,a_n)^{\phi(g_1)\phi(g_2)} = (a_1,\ldots,a_n)^{\phi(g_1g_2)}$ 
holds whenever $a_i^{\phi_i(g_1)\phi_i(g_2)}=a_i^{\phi_i(g_1g_2)}$ for each
$a_i$, which is the case for at least
$(1-\epsilon)^n|A|$ elements $(a_1,\ldots,a_n) \in A$. The result now follows,
since $(1-\epsilon)^n \ge 1-n\epsilon$ for all $n \ge 1$.
\end{proof}
 
This article contains two further sections;
Section~\ref{sec:graphprod} contains the proof of Theorem~\ref{thm:graphprod}
and 
Section~\ref{sec:graphgroups} the proof of Theorem~\ref{thm:graphgroups}.

\section{Proof of the graph product theorem}
\label{sec:graphprod}
Let $\Gamma$ be a simple graph and, for each vertex $v$ of $\Gamma$, let
$G_v$ be a group. The graph product of the groups $G_v$ with respect to
$\Gamma$ is defined to be the quotient of their free product by the normal
closure of the relators $[g_v,g_w]$ for all $g_v \in G_v$, $g_w \in G_w$
for which $\{v,w\}$ is an edge of $\Gamma$.

Graph products were introduced by Green in her PhD thesis \cite{Green},
and their basic properties are established there.
For a graph product of vertex groups $G_1,\ldots,G_n$ with respect to a finite
graph $\Gamma$ with vertices $1,\ldots,n$, and for $J \subseteq \{1,\ldots,n\}$, we define
$G_J := \langle G_j \mid j \in J \rangle$.
By \cite[Proposition 3.31]{Green}, $G_J$ is isomorphic to the graph product of
$G_j$ ($j \in J$) on the full subgraph of $\Gamma$ with vertex set $J$.
Note that $G_\emptyset$ is the trivial group.

Green only considered graph products of finitely many vertex groups, but
the definition applies equally well to graphs with infinite vertex sets $I$.
Since any relation in a group is a consequence of finitely many defining
relations, the property that, for any $J \subseteq I$, $G_J$ is isomorphic to
the graph product of $G_j$ ($j \in J$) on the full subgraph of $\Gamma$ with
vertex set $J$, extends to graph products with infinitely many vertex groups.
Hence, since a group is sofic if and only if all of its finitely generated
subgroups are sofic, it suffices to prove Theorem~\ref{thm:graphprod} for
graph products of finitely many groups, so we shall assume from now on that
the graph $\Gamma$ is finite.

Any non-identity element in a graph product can be written as a product
$g_1\cdots g_l$ for some $l>1$, where each $g_i$ is a non-trivial
element of a vertex group $G_{j_i}$.
By \cite[Theorem 3.9]{Green}, we can get from any such expression of minimal
length to any other by swapping the order in the expression of elements
$g_i$, $g_{i+1}$ from commuting vertex groups.
Hence every minimal length expression for an element $g$ has the same
length $l$, which we call the {\em syllable length} of $g$, and involves the
same set $\{g_1,g_2,\ldots g_l\}$ of vertex group elements, with the same
multiplicities, the
{\em syllables} of $g$.
Whenever $g_1 \cdots g_l$ is a minimal length expression for $g$, we call each
product $g_1 \cdots g_i$ a {\em left divisor} of $g$,
and each product $g_{i+1}\cdots g_n$ a {\em right divisor} of $g$, for $0 \le i \le n$.

We also note that, for any finite subset of a graph product
of groups $G_i$, there is a bound $N$ on the syllable lengths of its elements,
and there are finite subsets $F_i$ of the vertex groups $G_i$ that contain all
the syllables of those elements.
Hence  Theorem~\ref{thm:graphprod} follows from the following proposition. 

\begin{proposition}
There is a function $f:\N \to \N$ with the following property.
Let $G_1,\ldots,G_n$ be sofic groups, and $G$ their graph product with respect
to a finite graph $\Gamma$.
Let $\epsilon>0$ be given, and for each $i=1,\ldots n$,
let $F_i$ be a finite subset of $G_i$,  $A_i$ a finite set, and suppose that
$\psi_i:G_i \rightarrow \Sym(A_i)$ is a special $(F_i,\epsilon)$-quasi-action
of $G_i$ on $A_i$.

Then, for any $N \in \N$, $G$ has a special $(F,f(n)\epsilon)$-quasi-action
$\phi$ on a finite set $C$, where $F$ is the set of elements of $G$ of
syllable length at most $N$ for which each syllable is in some $F_i$,
such that the following additional properties hold:
\begin{mylist}
\item[(1)] whenever $x,y$ are in distinct vertex groups,
$\phi(xy)=\phi(x)\phi(y)$;
\item[(2)] $C$ admits
equivalence relations $\sim_1,\ldots,\sim_n$ such that,
for each $c \in C$, $g \in F$ and $J \subseteq \{1,\ldots, n\}$,
\[ c^{\phi(g)} \sim_J c \iff g \in G_J\]
(where $\sim_J$ is the join of those equivalence relations $\sim_j$ for which $j \in J$).
\end{mylist}
\end{proposition}

Note that, by definition, $a \sim_J b$ if and only if there is a sequence
$a=c_1,\ldots,c_m=b$ of elements
with $c_i \sim_{j_i} c_{i+1}$ for some $j_i \in J$.
In particular, $x \sim_\emptyset y \iff x=y$.

Note that the conditions (1) and (2)
imposed on the special quasi-action $\phi$ are necessary for the inductive proof
of the proposition, rather than to deduce the theorem. Condition (1) ensures in particular that
$\phi(x)\phi(y)=\phi(y)\phi(x)$ whenever $x,y$ are from commuting vertex groups.

\begin{proof}
The proof is by induction on $n$. Suppose first that $n=1$. Then $G=G_1$
and $F=F_1$ (for any value of $N \in \N$). We put $F := F_1$ and $C := A_1$,
and define the equivalence relation $\sim_1$ by $c \sim_1 d$ for all
$c,d \in C$. Then $\phi$ is a special $(F,\epsilon)$-quasi-action on $C$,
and the additional property (1) holds vacuously. To see that the additional
property (2) also holds, note that there are only two possibilities for $J$:
$J=\{1\}$ and $J = \emptyset$. If $J=\{1\}$ then $G=G_J$, so the left and
right hand sides of the equivalence in (2) are true for all $g \in G$.
If $J = \emptyset$ then, by the definition of a special
$(F,\epsilon)$-quasi-action, both the left and right hand sides of the
equivalence are true if and only if $g = 1$. So the property (2) holds,
and the statement of the proposition is true with $f(n)=1$.

So now we proceed to prove the inductive step.
We shall prove that the result holds with $f(n) = n(nf(n-1)+1)$.

Write $I=\{1,2,\ldots,n\}$, and for each $k \in I$, $I_k=I \setminus \{k\}$.
For each $k \in I$,
let $H_k := G_{I_k}$ be the subgroup of $G$ that is the graph
product of the groups $G_i$ for $i \neq k$ with respect to the appropriate subgraph of
$\Gamma$.  By the induction hypothesis, we may assume that,
for $\epsilon' := f(n-1)\epsilon$, and $F_{H_k} := F \cap H_k$,
$H_k$ has a special
$(F_{H_k},\epsilon')$-quasi-action $\theta_k$ on a set $D_k$ admitting 
equivalence
relations $\simeq^k_i$, for each $i \neq k$, such that
\begin{mylist}
\item[(1)] $\theta_k(xy)=\theta_k(x)\theta_k(y)$ for $x,y$ in distinct vertex
groups of $H_k$; and
\item[(2)]
for $d \in D_k$, $h \in F_{H_k}$, and $J \subseteq I_k$,
$d^{\theta_k(h)} \simeq_J d \iff h \in G_J$.
\end{mylist}

For each $k \in I$, we shall build a set $C_k$
related to $D_k$, admitting equivalence relations $\sim^k_i$ for each $i \in I$,
and then construct a special quasi-action $\phi_k$ of $G$ on $C_k$ that satisfies
Condition (1) and more. We shall then construct $\phi$ and the equivalence
relations $\sim_1,\ldots,\sim_n$ on the set $C := C_1 \times C_2 \times \cdots \times C_n$ in terms of the special quasi-actions $\phi_k$
and the equivalence relations $\sim^k_i$, using Lemma~\ref{lem:product}. 

For $k \in I$,
let $L_k \subseteq I_k$ be the set of vertices joined in $\Gamma$ to $k$.
Let $\simeq_{L_k}$ be the join of the equivalence relations $\simeq^k_i$ for $i \in L_k$, and let $\pi_k$ be
the projection from $D_k$ to its set of equivalence classes under $\simeq_{L_k}$
(for which the image of $d \in D_k$ is its equivalence class).

Now, using ideas from \cite[Theorem 1]{ElekSzabo}
we choose a finite group $V_k$,
with generating set $\pi_k(D_k)\times A_k$, for which all relators
among the generators have length greater than $N$, and we let
$C_k := D_k \times A_k \times V_k$.

We define equivalence relations $\sim^k_i$ on $C_k$, for $i \neq k$, by the
rules
\[ (d,a,v) \sim^k_i (d',a',v') \iff d \simeq^k_i d',\ a=a',\ v=v'.\]
Then we define $\sim^k_k$ on $C_k$ by specifying its equivalence classes;
for $d \in D_k, v \in V_k$; the class $\alpha_k(d,v)$ is the subset
$\{ (d,a,v\circ(\pi_k(d),a)): a \in A_k\}$ of $C_k$.
Multiplication $\circ$ within the third component is the group multiplication of $V_k$.

We define a special quasi-action $\phi_k$ of $G$ on $C_k$ as a composite
of natural extensions to $C_k$ of the special quasi-actions $\theta_k,\psi_k$ of $H_k$ and $G_k$ on $D_k$, $A_k$.

For $h \in H_k$, we define
\[ (d,a,v)^{\phi_k(h)} = (d^{\theta_k(h)},a,v).\]
Then, for $g \in G_k$, we define
\[ (d,a,v)^{\phi_k(g)} = (d,a^{\psi_k(g)},v\circ(\pi_k(d),a)^{-1}\circ(\pi_k(d),a^{\psi_k(g)})).\]
Now it follows, essentially from \cite[lemma 3.20]{Green}, that
each element $g \in G$ has a unique expression as a product
$g=x_1y_1\cdots x_my_m$, with each $x_i \in H_k$, each $y_i \in G_k$, $x_i$
nontrivial for $i>1$, $y_i$ nontrivial for $i<m$, and such that,
for $i>1$, $x_i$ has no non-trivial left divisor in the subgroup $G_{L_k}$; we
call this expression the {\em normal form} for $g$. We note that the $y_i$'s are
syllables, the $x_i$'s products of syllables and the number of terms at most
the syllable length of $g$. We use that expression for $g$ to extend 
to $G$ the definitions of $\phi_k$ on $H_k$ and $G_k$,
that is,
for $g \in G$,
$\phi_k(g):= \phi_k(x_1)\phi_k(y_1) \cdots \phi_k(x_m)\phi_k(y_m)$.

We need now the following lemma, whose proof we defer.

\begin{lemma}
\label{guts}
Let $\epsilon'' :=(nf(n-1)+1)\epsilon$.
Then, for each $k$, $\phi_k$ is a special $(F,\epsilon'')$-quasi-action of
$G$ on $C_k$, such that
\begin{mylist}
\item[(1)] whenever $x,y$ are in distinct vertex groups, $\phi_k(xy)=\phi_k(x)\phi_k(y)$,
\item[($2'$)] for each $c \in C_k$, $g \in F$, we have
$g \in G_J \Rightarrow c^{\phi_k(g)} \sim^k_J c$ for all $J \subseteq I$,
and $c^{\phi_k(g)}\sim^k_J c  \Rightarrow g \in G_J$ for all $J \subseteq I_k$.
\end{mylist}
\end{lemma}

Now we define a map $\phi:G \rightarrow \Sym(C)$, where
$C := C_1 \times \cdots \times C_n$, by
$(c_1,\ldots,c_n)^{\phi(g)} = (c_1^{\phi_1(g)},\ldots,c_n^{\phi_n(g)})$.
It follows from
Lemma~\ref{lem:product} that this is a
$(F,f(n)\epsilon)$-quasi-action with
$f(n) = n(nf(n-1)+1)$.
Condition (1) of the proposition is inherited from the maps $\phi_k$.

We define equivalence relations $\sim_{1},\sim_{2},\ldots,\sim_{n}$
on $C$ by $(c_1,\ldots,c_n) \sim_j (c'_1,\ldots,c'_n)$ if and only if
$c_k \sim^k_j c'_k$ for $1 \le k \le n$.
We need now to verify Condition (2).

Let $J \subseteq I$. The fact that $g \in G_J$ implies that
$c^{\phi(g)} \sim_J c$ for all $c \in C$ is inherited from
the maps $\phi_k$.
If $J = I$, then $G=G_J$ and the converse statement is immediate.
Otherwise we have $J \subseteq I_k$ for some $k$ with $1 \le k \le n$.
If $g \not\in G_J$, and $c = (c_1,\ldots,c_n) \in C$, then
$c_k^{\phi_k(g)} \not\sim^k_J c_k$ and hence $c^{\phi(g)} \not\sim_J c$.

So the proof of the proposition will be complete once the proof of
Lemma~\ref{guts} has been provided. 
\end{proof}

\begin{proofof}{Lemma~\ref{guts}}
Note that
it is  clear that the restriction of $\phi_k$
to $H_k$ is a
special $(F_{H_k},\epsilon')$-quasi-action for $H_k$,
since $\theta_k$ is.
And certainly that quasi-action preserves each of the $\sim^k_i$ equivalence classes with $i\neq k$.  
And it is clear that the restriction of $\phi_k$ to $G_k$
is a special
$(F \cap G_k ,\epsilon)$-quasi-action for $G_k$, since $\psi_k$ is.
That quasi-action preserves the $\sim^k_k$ equivalence classes,
since both $(d,a,v)$ and $(d,a,v)^{\phi_k(g)}$ are in $\alpha_k(d,v\circ(\pi_k(d),a)^{-1})$.

The equation $(d,a,v)^{\phi_k(1_G)} = (d,a,v)$ follows immediately from
$(d,a,v)^{\phi_k(h)} = (d^{\theta_k(h)},a,v),$ for $h \in H_k$, and hence Condition (a) of Definition~\ref{defn:qaction} is verified for $\phi_k$.

We shall verify the remaining conditions in the order
(c), (1), (b), (d), ($2'$).

First we introduce some notation.
We need to consider $\phi_k(g)$ for a general element $g$ in the graph product,
written in normal form as $x_1y_1\cdots x_my_m$.
We write $x$ for the group product $x_1\cdots x_m$,
then $y$ for the group product
$y_1\cdots y_m$, and $x[i],y[i]$ for the products $x_1\cdots x_i$,
$y_1,\cdots y_i$, where $x[0] = y[0] = 1$. 

We see then that 
\begin{eqnarray*}
(d,a,v)^{\phi_k(g)}&=&(d,a,v)^{\phi_k(x_1y_1\cdots x_my_m)}
= (d^{\theta_k(x[m])},a^{\psi_k(y[m])},v \circ u),\\
\quad\hbox{\rm where} \quad u&=&\prod_{i=1}^m (\pi_k(d^{\theta_k(x[i])}),a^{\psi_k(y[i-1])})^{-1}\circ
(\pi_k(d^{\theta_k(x[i])}),a^{\psi_k(y[i])}).
\end{eqnarray*}
unless $y_m$ is the identity, in which case the product for $u$ is from $i=1$
to $m-1$.

Our next step is to establish Condition (c) of Definition~\ref{defn:qaction}
for $\phi_k$.
Let $g$ be a non-trivial element of $F$, with normal form
$x_1y_1 \cdots x_m y_m$.
So $2m \le N$ and, for each $i$, $x_i \in F_{H_k}$ and $y_i \in F_k$.
Suppose first that $u$, in the above expression, is not the
empty word.  Since $\psi_k$ is a special quasi-action,
Condition (c) for $\psi_k$ implies that
$a^{\psi_k(y[i-1])}\neq a^{\psi_k(y[i])}$ for each $i$. 
Since $x_{i+1} \not \in G_{L_k}$, it follows from the induction hypothesis that
$\theta_k(x_{i+1})$ cannot map any element of $D_k$ to an element in the same $\simeq_{L_k}$
equivalence class, that is, $\pi_k(d^{\theta_k(x[i])}) \neq \pi_k(d^{\theta_k(x[i+1])})$. 
So no generator
in the word of length $2m$ representing $u$ can freely cancel with the
generator either before it or after it. 
The fact that $V$ admits no short relators now ensures that $u$ is nontrivial.
In that case certainly $(d,a,v)^{\phi_k(g)} \neq (d,a,v)$.

So now suppose that $u$ is empty.
Then $m=1$, $y_1$ is trivial, and $g=x_1$. So $x=x_1$ is a non-identity
element of $F_{H_k}$, and hence $d^{\theta_k(x)} \neq d$. So again $(d,a,v)^{\phi_k(g)} \neq (d,a,v)$.
Hence we have shown that the map $\phi_k$ from $G$ to $\Sym(C_k)$ allows no
non-identity element of length less than $N$ in $F$ to fix any element of $C_k$,
and so Condition (c) of Definition~\ref{defn:qaction} is verified for $\phi_k$.

In order to establish Condition (1) of the Lemma for $\phi_k$,
we suppose first that $x \in G_{L_k}$, and $y \in G_k$. 
By definition $\phi_k(xy)=\phi_k(x)\phi_k(y)$, and
\[ (d,a,v)^{\phi_k(x)\phi_k(y)} = (d^{\theta_k(x)},a^{\psi_k(y)},v \circ
(\pi_k(d^{\theta_k(x)}),a)^{-1}\circ(\pi_k(d^{\theta_k(x)}),a^{\psi_k(y)}))\]
while
\begin{eqnarray*}
(d,a,v)^{\phi_k(y)\phi_k(x)} &=& (d,a^{\psi_k(y)},v\circ(\pi_k(d),a)^{-1})\circ(\pi_k(d),a^{\psi_k(y)}))^{\phi_k(x)} \\
&=& (d^{\theta_k(x)},a^{\psi_k(y)},v\circ (\pi_k(d),a)^{-1}\circ (\pi_k(d),a^{\psi_k(y)})).\end{eqnarray*}
Then since $d \simeq_{L_k} d^{\theta_k(x)}$, we have $\pi_k(d) = \pi_k(d^{\theta_k(x)})$,
and so \[(d,a,v)^{\phi_k(x)\phi_k(y)} = (d,a,v)^{\phi_k(y)\phi_k(x)},\]
that is, for $x \in G_{L_k}$, $y \in H_k$,
$\phi_k(xy)=\phi_k(x)\phi_k(y) = \phi_k(y)\phi_k(x)$.

Now suppose that $x,y$ are
in distinct vertex groups, $G_i,G_j$.
If
$i,j \neq k$ then Condition (1) follows immediately by induction applied to $H_k$.
If $j=k$, or if $i=k$ and $G_i,G_j$ do not commute, then $xy$ is in normal
form, and Condition (1) follows from the definition of $\phi_k$.
Finally if $i=k$ and $G_i,G_j$ commute, then $x \in G_{L_k}, y \in H_k$,
and we can deduce Condition (1) for $\phi_k$ from the result above.

Next suppose that $g=x_1y_1\cdots x_my_m \in G$.
We compare $\phi_k(g)^{-1}$ and $\phi_k(g^{-1})$.
We have $g^{-1}=y_m^{-1}x_m^{-1}\cdots y_1^{-1}x_1^{-1}$.
The expression for $g^{-1}$ is not necessarily in normal form, because some of
the $x_i^{-1}$ could have left divisors in $G_{L_k}$, but we can transform it
into normal form by splitting any such $x_i^{-1}$ into syllables and
then applying commuting relations to move left divisors of $x_i^{-1}$ in
$G_{L_k}$ past $y_i^{-1}$. By the results of the preceding two paragraphs, if
we apply the corresponding transformations to
$\phi_k(y_m^{-1})\phi_k(x_m^{-1})\cdots \phi_k(y_1^{-1})\phi_k(x_1^{-1}$), then
we do not change the resulting permutation.
Hence we have
$\phi_k(g^{-1}) =
\phi_k(y_m^{-1})\phi_k(x_m^{-1})\cdots \phi_k(y_1^{-1})\phi_k(x_1)^{-1}$.
It follows from Condition (b) of Definition~\ref{defn:qaction} that
$\phi_k(y_i^{-1})$ is inverse to $\phi_k(y_i)$
and from the induction hypothesis on $H_k$ that $\phi_k(x_i^{-1})$ is inverse to
$\phi_k(x_i)$.
Hence $\phi_k(g^{-1}) = \phi_k(g)^{-1}$, which verifies Condition (b) of
Definition~\ref{defn:qaction} for $\phi_k$.

We proceed now to verify Condition (d) of Definition~\ref{defn:qaction} for
$\phi_k$; that is, to show that for all $g_1,g_2 \in F$, $\phi_k(g_1g_2)$ is
$\epsilon''$-similar to $\phi_k(g_1)\phi_k(g_2)$.
Let $g_1=x_1y_1\cdots x_my_m$, $g_2=x'_1y'_1\cdots x'_{p}y'_{p}$
be the normal forms of $g_1,g_2 \in F$.
In the following discussion, we refer to an element of $H_k$ or of $G_k$ as
a {\em block}, and to a product of blocks as an {\em expression}.
The normal form for $g_1g_2$ is derived from the concatenation
$x_1y_1\cdots x_my_mx'_1y'_1\cdots x'_{p}y'_{p}$ by a sequence of moves,
each of which is one of four types:

\begin{description}
\item[(a)] deletion of a block that is equal to the identity;
\item[(b)] cancellation (that is, merger of two adjacent mutually
inverse blocks that are either both in $H_k$ or both in $G_k$);
\item[(c)] expression of a block in $H$ as a product of a left divisor in
$G_{L_k}$ and a right divisor, and moving the left divisor to the left,
past a block in $G_k$;
\item[(d)] merger of two adjacent blocks that are either both in $H_k$ or
both in $G_k$, and whose product is not the identity,
to give a new block from that same subgroup.
\end{description}
Note that in (c) the left and right divisors of a block in $H_k$ are simply
subblocks, whose concatenation is a permutation of the original block; that is,
the (multi)set of syllables of the block in $H_k$ is the union of the
(multi)sets of syllables of those left and right divisors.
By contrast, a move of type (d) will normally
change the (multi)set of syllables in an expression.
Starting with the permutation
$$\phi(x_1)\phi(y_1)\cdots \phi(x_m)\phi(y_m)\phi(x'_1)\phi(y'_1)\cdots
\phi(x'_{p})\phi(y'_{p}),$$
we study the sequence of composites of permutations of $C_k$ defined by
the various expressions that arise when we apply the corresponding
operations to this expression of images during this rewrite process,
and keep track of the proportion of elements of $C_k$ on which they differ.
We note that, as a consequence of what we have proved so far,
two expressions that differ only on moves of types (a), (b) and (c)
correspond to composites of permutations that have the same effect on all
points of $C$. Hence we only need to concern ourselves with moves of type (d).

Suppose that a move converts an expression $w$ to an expression $w'$.
Let $\sigma, \sigma'$ be the permutations corresponding to the two expressions.
If the move merges two blocks from $G_k$, then the permutations $\sigma$
and $\sigma'$ differ on the same proportion of elements of $C_k$ as do permutations
for the quasi-action of $G_k$ on the set $A_k$, that is, on at most
$\epsilon|C_k|$ of the elements, by the hypothesis.

If the move merges two blocks from $H_k$, then the permutations $\sigma$ and
$\sigma'$ differ on the same proportion of elements of $C_k$ as do permutations
for the quasi-action of $H_k$ on the set $D_k$, that is, on at most
$f(n-1)\epsilon|C_k|$ of the elements, by the induction hypothesis.
Notice however that if the two blocks $z_1,z_2$ being merged are left and
right divisors of $z_1z_2$ (or, equivalently, if the syllable length of
$z_1z_2$ is the sum of the syllable lengths of $z_1$ and $z_2$),
then our induction hypothesis on $H$ ensures that
$\phi_k(z_1z_2) = \phi_k(z_1)\phi_k(z_2)$.
We shall call such mergers {\it non-reducing}, and other mergers,
for which this equality is not guaranteed to hold, {\it reducing}.

Condition (d) can now be now established by application of the following lemma.

\begin{lemma}
During the rewrite process, we perform at most $n$ reducing mergers of
blocks of $H_k$ and at most one reducing merger of blocks of $G_k$.
\end{lemma}
\begin{proof}
We may assume that $m,p>0$ (since otherwise one of $g_1$, $g_2$ is the identity)
and split the proof into three cases
(1) $1 \neq y_m$ and $x'_1 \not\in G_{L_k}$;
(2)  $y_m=1$; and (3) $1 \neq y_m$ and $x_1 \in G_{L_k}$.

We deal with Case 1 first, proving by induction on $m$ that in this case
the product can be 
rewritten using at most $|L_k|$ mergers, all of which are  within $H_k$.
Using that result we then deal with the remaining two cases together,
also using induction on $m$.

{\bf Case 1.} $1 \ne y_m$ and $x'_1 \not\in G_{L_k}$.

Let $x'_1 = z_1z_2$, where $z_1$ is the longest left divisor of $x'_1$ in
$G_{L_k}$.  Suppose that $z_1 \in G_{L'}$ for some $L' \subseteq L_k$.
We prove by induction on $m$ that this product can be rewritten using
at most $|L'|$ ($\leq |L_k|$) $H_k$-mergers and no $G_k$-mergers.

If $m = 1$ then there can be at most one $H_k$-merger
$x_1x'_1$, so the result is clear.
So suppose that $m>1$; then $y_{m-1} \ne 1$, and $x_m$ is
nontrivial with no left divisor in $G_L$. If $z_1$ commutes with $x_m$, then
the claim follows by induction applied to the product
$(x_1y_1 \cdots x_{m-1}y_{m-1})(z_1x_my_mz_2y'_1 \cdots x'_{p}y'_{p})$.
Otherwise, we can write $z_1 = z_{11}z_{12}$, where $z_{11}$ (which may be
trivial) is the longest left divisor of $z_1$ that commutes with $x_m$.  So
$z_{11} \in G_{L''}$ with $|L''| < |L'|$.  We can then perform the rewriting by
performing an $H_k$-merger $x_mz_{12}$ (if necessary) and, by induction, at most
$|L''|$ further $H_k$-mergers resulting from moving $z_{11}$ further to the
left. This completes the proof of the claim, and of the lemma in Case 1.

So now we may assume that $m,p>0$, and that we are in Case 2 or 3.

{\bf Case 2.} $y_m=1$.

If $m=1$, then there is at most one $H_k$-merger $x_1x'_1$, and so the result
holds.
So suppose that $m>1$ and hence that $y_{m-1} \ne 1$, and $x_m$ is
nontrivial with no left divisor in $G_{L_k}$. 

If $x_mx'_1 \not\in G_{L_k}$, then we
perform an $H_k$-merger (if necessary) on $x_mx'_1$, and
now observe that the product
$(x_1y_1 \cdots x_{m-1}y_{m-1})(x_mx'_1 y'_1 \cdots x'_{p}y'_{p})$
satisfies the conditions of Case 1,
and so can be rewritten using at most $|L_k|$ further $H_k$-mergers and no
$G_k$-mergers. So in this case too, the lemma is proved.

If $x_mx'_1 \in G_{L_k}$ then, since $x_m$ has no left divisor in $G_{L_k}$, the
product $x_mx'_1$ can be evaluated by writing $x_m$ and $x'_1$ as products
of syllables and then performing commuting and cancellation moves only
so we can rewrite $x_mx'_1$ as $z \in G_{L_k}$
without performing any mergers, to arrive at the product
\[(x_1y_1 \cdots x_{m-1}y_{m-1})(z y'_1 \cdots x'_{p}y'_{p}),\]
which satisfies the conditions of Case 3 for $m-1$.
The lemma now follows by induction applied to that product.

{\bf Case 3.} $1 \ne y_m$ and $x'_1 \in G_{L_k}$.

If $y_m y'_1 \ne 1$, then we perform the $G_k$-merger $y_m y'_1$,
and the $H_k$-merger $x_mx'_1$ (which cannot be in $G_{L_k}$, since
$x_m \not\in G_{L_k}, x'_1 \in G_{L_k}$).
Then we can apply the result of Case 1 to the product
$(x_1y_1 \cdots x_{m-1}y_{m-1})(x_mx'_1 y_my'_1 \cdots x'_{p}y'_{p})$,
and the proof is complete.

If $y_m y'_1 = 1$ then the result is clear if $p=1$ and otherwise, since
$x'_2$ has no left divisor in $G_L$, the merger $x'_1x'_2$ is non-reducing,
so the result follows by applying Case 2 to the product
$(x_1y_1 \cdots x_{m-1}y_{m-1} x_m)(x'_1x'_2 y'_2\cdots x'_{p}y'_{p})$.
\end{proof}

This completes the proof of Condition (d), and hence we see that
$\phi_k$ is a special
$(F,\epsilon'')$-quasi-action, with $\epsilon'' = (nf(n-1) + 1)\epsilon$.

It remains to verify Condition ($2'$).
We have shown already that, for
each $i \in I$, the action of $\phi_k(G_i)$ on $C$ preserves each of the
$\sim^k_i$-equivalence classes, from which it follows immediately that
$g \in G_J$ with $J \subseteq I$ implies $c^{\phi_k(g)} \sim^k_J c$.

Now suppose that $J \subseteq I_k$,  
$c=(d,a,v) \in C_k$, $g \in F$, and that $c^{\phi_k(g)} \sim^k_J c$.
Since $k \not \in J$, it is immediate from the definition of $\sim^k_j$ for
$j \in J$ that
\[ (d,a,v) \sim^k_J (d',a',v') \iff d \simeq_J d',\ a=a',\ v=v'. \]
So now, arguing as in our earlier proof of Condition (c) of Definition~\ref{defn:qaction} for $\phi_k$ that, for $1 \ne g \in F$,
$(d^{\theta_k(g)},a,v) \ne (d,a,v)$ we find that, for $g \in F$, $(d^{\theta_k(g)},a,v) \sim_J (d,a,v)$
if and only if $g \in H_k$ and $d^{\theta_k(g)} \simeq_J d$. By
our inductive hypothesis, this is true if and only if $g \in G_J$.
Hence Condition ($2'$) holds.
\end{proofof}

\section{Graphs of groups}
\label{sec:graphgroups}

In this section we prove Theorem~\ref{thm:graphgroups}.

We recall the definition of a graph of groups, which arises from the work
of Bass and Serre \cite{SerreFrench,Serre}
\begin{definition}
A graph of groups $\mathcal{G}$ consists of 
\begin{mylist}
\item[(1)] a connected graph $\Gamma$ (in which loops are allowed, but no multiple edges), with vertex set $V$, edge set $E$,
\item[(2)] a collection of {\em vertex groups} $G_v: v \in V$ and
{\em edge groups} $G_e: e \in E$,
\item[(3)] for each edge $e=\{v_1,v_2\}$ of $\Gamma$, monomorphisms $\theta^1_e:G_e \rightarrow G_{v_1}$ and $\theta^2_e:G_e \rightarrow G_{v_2}$.
\end{mylist}
\end{definition}

The {\em fundamental group} $\pi_1(\mathcal{G})$ of a graph of groups $\mathcal{G}$ can defined in various different (but equivalent) ways. The following definition
is essentially \cite[Definition I.3.4]{DicksDunwoody}.
The definition is given in terms of a selected spanning tree $T$ of $\Gamma$,
but (up to isomorphism) the resulting group is independent
of this choice.
The associated fundamental group $\pi_1(\mathcal{G},T)$ is then the group generated by
the groups $G_v: v \in V$ together with generators $t_e$, one for each (oriented)
edge in $E$), given the following relations.
\begin{mylist}
\item[(1)] all the relations of the groups $G_v$,
\item[(2)] $t_e^{-1}\theta^1_e(g)t_e = \theta^2_e(g)$, for each $e \in E,g\in G_e$,
\item[(3)] $t_e = 1$ for each edge $e$ of $T$.
\end{mylist}
From this description it is not hard to see that $\pi_1(\mathcal{G},T)$
is isomorphic to a multiple HNN extension, with stable letters $t_e$ for $e \not \in E(T)$, of the amalgamated product of the groups $G_v$ in which
$\theta^1_e(g)$ and $\theta^2_e(g)$ are identified for all $e \in E(T), g \in G_e$.
Independent results of Elek and Szabo (\cite[Theorem 1]{ElekSzabo2}) and Paunescu
(\cite[Corollary 2.3]{Paunescu}) already prove that the 
amalgamated product of two sofic groups over an amenable subgroup is sofic.
Hence Theorem~\ref{thm:graphgroups} follows immediately
by combining that result with 

\begin{proposition}\label{HNN}
An HNN extension of a sofic group $H$ over an amenable subgroup $K$ is sofic.
\end{proposition}  
We deduce Proposition~\ref{HNN} as a corollary of the amalgamated
product result.
We note that the argument to do this was already provided by Collins and Dykema
in order to deduce their result \cite[Corollary 3.6]{CollinsDykema}
as a corollary of their result \cite[Theorem 3.4]{CollinsDykema},
that is to deduce the same result as above in the situation where the
associated subgroups (in both amalgamated products and HNN extensions)
are monotileably amenable.  
The argument of \cite{CollinsDykema} goes through without any modification, when
monotileability of the associated subgroup is dropped, 
to deduce the Proposition from the results of \cite{ElekSzabo2,Paunescu}.
But we include the argument here for completeness.

\begin{proof}
Let $G$ be an HNN extension of $H$ over $K$, as in the proposition,
and let $L$ be the subgroup $t^{-1}Kt$.
Define $H_i=t^{-i} H t^i$, $K_i=t^{-i} K t^i$,
$L_i=t^{-i} L t^i$ for each $i \in \Z$, and define
$S:= \langle H_i \mid i \in \Z \rangle$.
Then $G$ can be expressed as an extension of $S$ by $\Z$.
Since $\Z$ is amenable, and by \cite[Theorem 1(3)]{ElekSzabo} an extension
of a sofic group by an amenable group is sofic, in order to prove $G$
sofic it is enough to prove $S$ sofic.

Now $S$ can be expressed as an iterated amalgamated product of the (countably many)  $H_i$s, 
with amalgamation over subgroups isomorphic to $K$. More precisely, $S$ is the
fundamental group of the graph of groups associated with the graph of the
integers, where $H_i$ is the vertex group of the vertex $i$, each edge group is
isomorphic to $K$, and the copy of $K$ associated with edge $\{i,i+1\}$ maps to
the subgroup $L_i$ of $H_i$, and the subgroup 
$K_{i+1}$ of $H_{i+1}$, as in Figure 1.  

\begin{figure}[!h]
\centerline{
\xymatrix{
\cdots H_{i-1}\ar@{-}[rr]^{\ L_{i-1}\hookleftarrow \ \ \hookrightarrow K_i} && H_i \ar@{-}[rr]^{L_{i}\hookleftarrow \ \ \hookrightarrow K_{i+1}} && H_{i+1} \ar@{-}[rr]^{L_{i+1}\hookleftarrow \ \hookrightarrow K_{i+2}} && H_{i+2} \cdots \\
}
}
\caption{The graph of groups $\mathcal{H}$}
\end{figure}

To prove $S$ sofic we now need to verify soficity for each of its finitely
generated subgroups. So let $M$ be such a subgroup. Then for some $k,l$,
all the generators of $M$ are within vertex subgroups $H_i$
for $k \leq i \leq l$,   
that is, $M$ is a subgroup of the amalgamated product
\begin{equation*}
H_j \ast_{\substack{L_j=K_{j+1}}} H_{j+1} \ast_{L_{j+1}=K_{j+2}} \ast \cdots \ast_{L_{l-1}=K_{l}} H_l.
\end{equation*}
Since this is sofic, by \cite{ElekSzabo2,Paunescu}, so is $M$. 
\end{proof}

\section*{Acknowledgments}
All three authors were partially supported by the Marie Curie Reintegration
Grant 230889. The first named author was also supported by the Swiss National Science Foundation grant Ambizione PZ00P-136897/1.

\textsc{L. Ciobanu,
Mathematics Department,
University of Neuch\^atel,
Rue Emile - Argand 11,
CH-2000 Neuch\^atel, Switzerland
}

\emph{E-mail address}{:\;\;}\texttt{laura.ciobanu@unine.ch}
\medskip

\bigskip

\textsc{D. F. Holt,
Mathematics Institute,
University of Warwick,
Coventry CV4 7AL,
UK
}

\emph{E-mail address}{:\;\;}\texttt{D.F.Holt@warwick.ac.uk}
\medskip
\bigskip

\textsc{S. Rees,
School of Mathematics and Statistics,
University of Newcastle,
Newcastle NE1 7RU,
UK
}

\emph{E-mail address}{:\;\;}\texttt{Sarah.Rees@newcastle.ac.uk}
\medskip


\begin{thebibliography}{1}

\bibitem{CollinsDykema} B. Collins and K. J. Dykema, Free products of sofic
groups with amalgamation over monotileably amenable groups,
M\"unster J. Math. 4 (2011), 101--118.
\bibitem{DicksDunwoody} W. Dicks and M.L.Dunwoody, Groups acting on graphs, Cambridge studies in advanced mathematics 17, C.U.P. 1989.
\bibitem{ElekSzabo} G. Elek and E. Szabo, On sofic groups, J. Group Theory 9
(2006), 161 --171.
\bibitem{ElekSzabo2} G. Elek and E. Szabo, Sofic representations of amenable
groups, Proc. Amer. Math. Soc. 139 (2011), 4285--4291.
\bibitem{Green} E. Green, Graph products of groups, Ph.D. thesis, University of Leeds, 1990.
\bibitem{Gromov} M. Gromov, Endomorphisms of symbolic algebraic varieties,
J. Eur. Math. Soc (JEMS) 1 (1999) 109--197.
\bibitem{Paunescu} L. Paunescu, On sofic actions and equivalence relations,
J. Funct. Anal. 261 (2011) 2461–-2485.
\bibitem{Pestov} V.G. Pestov, Hyperlinear and sofic groups: a brief guide, Bull. Symbolic Logic 14 (2008), 449--480.
\bibitem{PestovKwiatowska} V.G. Pestov and A. Kwiatkowska, An introduction to
hyperlinear and sofic groups (preprint) http://arxiv.org/pdf/0911.4266.
\bibitem{Serre} J.-P. Serre, Trees, Springer Monogr. Math., Springer, Berlin 2003.
\bibitem{SerreFrench} J.-P.Serre, Arbres, amalgames, $SL_2$, Ast\'erisque 46 (1977).
\bibitem{Weiss} B. Weiss, Sofic groups and dynamical systems, Ergodic Theory and
Harmonic Analysis 2000 (62) 350--359.
\end{thebibliography}
\end{document}